\documentclass[11pt]{amsart}
\usepackage{amsfonts}
\usepackage{amsmath}
\usepackage{amsthm}
\usepackage{url}
\usepackage[all]{xy}
\usepackage{latexsym}
\usepackage{amssymb}
\usepackage[cp850]{inputenc}
\usepackage{amsfonts}
\usepackage{graphicx}

\usepackage[margin=1.5in]{geometry}
\usepackage[usenames]{color}

\newtheorem{sat}{Theorem}[section]		
\newtheorem{lem}[sat]{Lemma}
\newtheorem{kor}[sat]{Corollary}			
\newtheorem{prop}[sat]{Proposition}

\newtheorem*{defi*}{Definition}			
\newtheorem*{bei*}{Example}
\newtheorem*{sat*}{Theorem}				
\newtheorem*{kor*}{Corollary}
\newtheorem*{rmk*}{Remark}				
	
\newtheorem*{quest*}{Question}

\let\ssection=\section
\renewcommand{\section}{\setcounter{equation}{0}\ssection}

\newtheorem*{namedtheorem}{\theoremname}
\newcommand{\theoremname}{testing}
\newenvironment{named}[1]{\renewcommand{\theoremname}{#1}\begin{namedtheorem}}{\end{namedtheorem}}
\newtheorem*{bem}{Remark}

\theoremstyle{remark}

\newtheorem*{namedtheoremr}{\theoremnamer}
\newcommand{\theoremnamer}{testing}

\newcommand{\BR}{\mathbb R}			
\newcommand{\BN}{\mathbb N}			
			\newcommand{\BZ}{\mathbb Z}

\newcommand{\CC}{\mathcal C}		
		
\newcommand{\CG}{\mathcal G}		
		
		\newcommand{\CL}{\mathcal L}
\newcommand{\CM}{\mathcal M}

\newcommand{\CS}{\mathcal S}		\newcommand{\CT}{\mathcal T}

\newcommand{\actson}{\curvearrowright}
\newcommand{\D}{\partial}

\DeclareMathOperator{\vol}{vol}		%	Volumen

\DeclareMathOperator{\Map}{Map}

\newcommand{\comment}[1]{}

\DeclareMathOperator{\Homeo}{Homeo}

\DeclareMathOperator{\Lip}{Lip}
\DeclareMathOperator{\Thu}{Thu}

\DeclareMathOperator{\supp}{supp}

\DeclareMathOperator{\I}{i}

\DeclareMathOperator{\WP}{wp}

\newcommand{\param}{{\mathchoice{\mkern1mu\mbox{\raise2.2pt\hbox{$
\centerdot$}}
\mkern1mu}{\mkern1mu\mbox{\raise2.2pt\hbox{$\centerdot$}}\mkern1mu}{
\mkern1.5mu\centerdot\mkern1.5mu}{\mkern1.5mu\centerdot\mkern1.5mu}}}

\newcommand{\Teich}{{Teichm\"uller }} 
\newcommand{\from}{\colon\thinspace}

\newcommand{\fsubd}{\mathrel{{\scriptstyle\searrow}\kern-1ex^d\kern0.5ex}}
\newcommand{\bsubd}{\mathrel{{\scriptstyle\swarrow}\kern-1.6ex^d\kern0.8ex}}

\newcommand{\ML}{{\CM\CL}}

\newcommand{\Measure}{\mathfrak{M}}

\newcommand{\bSigma}{{\overline \Sigma}}
\newcommand{\st}{{\,\big\vert\,}}

\newcommand{\m}{{\bf m}}

\begin{document}

\title[Currents and Counting]{Geodesics Currents and Counting Problems}
\author{Kasra Rafi}
\address{Department of Mathematics, University of Toronto}
\email{rafi@math.toronto.edu}
\author{Juan Souto}
\address{IRMAR, Universit\'e de Rennes 1}
\email{juan.souto@univ-rennes1.fr}

\thanks{This material is based upon work supported by the National Science Foundation under Grant No. DMS-1440140 while the authors were in residence at the Mathematical Sciences Research Institute in Berkeley, California, during the Fall 2016 semester.
The first author was also partially supported by NSERC Discovery grant, RGPIN 435885.}

\begin{abstract}
For every positive, continuous and homogeneous function $f$ on the space of currents on 
a compact surface $\bSigma$, and for every compactly supported filling current $\alpha$, 
we compute as $L \to \infty$, the number of mapping classes $\phi$ so that 
$f(\phi(\alpha))\leq L$. As an application, when the surface in question is closed, 
we prove a lattice counting theorem for \Teich space equipped with the Thurston metric. 
\end{abstract}
\maketitle

\section{Introduction}
%
%$$\frac 1{\sum_{k=2}^{n-2}{n-2\choose k}{n-2\choose k-2}}\sum_{k=2}^{n-2}(k^2-n^2/4){n-2\choose k}{n-2\choose k-2}$$
%
%
Let $\bSigma$ be a compact orientable surface of negative Euler-characteristic and denote by $\CC=\CC(\bSigma)$ the associated space of currents, endowed it with the weak-*-topology (see \S \ref{Sec:Prelim}). The mapping class group 
$$\Map(\Sigma)=\Homeo_+(\Sigma)/\Homeo_0(\Sigma)=\Homeo_+(\bSigma)/\Homeo_0(\bSigma)$$ 
of the interior $\Sigma=\bSigma\setminus\D\bSigma$  of $\bSigma$ acts on $\CC$ by homeomorphism and the embedding $\ML(\Sigma)\to\CC(\bSigma)$ of the space of measured laminations into $\CC$ is mapping class group equivariant. In particular we can push forward the Thurston measure $\mu_{\Thu}$ on $\ML=\ML(\Sigma)$ to a mapping class group invariant measure on $\CC$.

For a continuous and homogeneous map $f \from \CC \to \BR_+$, define
\[
m(f) = \mu_{\Thu} \big( \big\{ \lambda \in \ML  \st f(\lambda) \leq 1 \big\}\big). 
\] 
We will be only interested in functions $f$ which are positive on the subset $\ML$. For such functions, the set $\{ \lambda \in \ML  \st f(\lambda) \leq 1\}$ is compact and thus $m(f)<\infty$. An important instance of such function is $f_\alpha(\param) = \I(\alpha, \param)$
where $\I \from \CC \times \CC \to \BR$ is the intersection pairing and 
$\alpha$ is a \emph{filling} current 
(ensuring $f_\alpha$ is positive on $\ML$). Relaxing notation, we write for a filling current $\alpha$
\[
m(\alpha)=\mu_{\Thu}\left(\big\{\lambda\in\CM\CL \st 
  \I(\alpha,\lambda)\le 1\big\}\right)
\] 
instead of $m(f_\alpha)$. Other important examples of positive, homogenous, and continuous functions on the space of currents arise as extensions of lengths functions, see \cite{Hugo}. In the particular case that the length function comes from a point $X\in\CT(\Sigma)$ in Teichm\"uller space then we set
 $$m(X)=\mu_{\Thu}\left(\big\{\lambda\in\CM\CL \st \ell_X(\lambda)\le 1\big\}\right).$$
The function $X\to m(X)$ is invariant under the action of the mapping class group on Teichm\"uller space and thus descends to a function on moduli space $\CM=\CM(\Sigma)$. Let 
\[
\m_\Sigma = \int_{\CM(\Sigma)} m(X) \, d \!\vol_{\WP},
\]
be the integral of that function with respect to the Weil-Petersson volume form.

With all this notation in place we are ready to state our main theorem, a general counting statement for the number of mapping class group orbits of a given current.

\begin{named}{Main Theorem}
Let $\bSigma$ be a compact surface of genus $g$ and $n$ boundary components, and let $\Sigma=\bSigma\setminus\D\Sigma$ be its interior. For a positive, continuous and homogeneous map $f \from \CC(\bSigma) \to \BR_+$ and a filling compactly supported current $\alpha \in \CC_c(\bSigma)$, we have
\[
  \lim_{L\to \infty} 
  \frac{ \# \big\{ \phi \in \Map(\Sigma) \st f\big(\phi(\alpha)\big) \leq L \big\}}{L^{6g+2n-6}}= 
  \frac{m(\alpha) \, m(f)}{\m_\Sigma}. 
\]
\end{named}

We can view the Main Theorem as an extension of an earlier result due to Mirzakhani. 
In \cite{M} she proved that if $X$ is a hyperbolic surface and $\gamma$ a filling closed 
curve, then 
\begin{equation}\label{Maryam}
  \lim_{L\to \infty} 
  \frac{\# \big\{ \phi \in \Map(\Sigma) \st \ell_X(\phi(\gamma)) \leq L \big\}}{L^{6g+2n-6}}
   = \frac{m(X) \, n_\gamma}{\m_\Sigma}
\end{equation}
without providing the explicit value for $n_\gamma$. In fact, Mirzakhani's result is used as key step in our proof and in section \ref{WP} we will show that $n_\gamma=m(\gamma)$.

\subsection*{Lattice counting in \Teich space}
The Main Theorem can be applied to other problems besides counting curves. As a perhaps unexpected application we prove a lattice counting theorem in \Teich space. Recall first that in \cite{ABEM}, Athreya, Bufetov, Eskin and Mirzakhani proved that for any two $X,Y\in\CT$ the number of orbit points $\Map(\Sigma)\cdot Y$ within Teichm\"uller distance $R$ of $X$ is asymptotic, as $R$ tends to $\infty$, to a constant multiple of $e^{\dim\CT\cdot R}$.
We establish the analogous result when we endow \Teich space with the Thurston metric
\[
d_{\Thu}(X,Y)=\log\left(\inf_{h: X\to Y}\Lip(h)\right).
\]
Here $\Lip(h)$ is the Lipschitz constant of $h$ and the infimum is taken over all Lipschitz maps $h \from X \to Y$ in the right homotopy class. For $X \in \CT$ and $\eta \in \CC$, define
\begin{equation} \label{DX} 
D_X(\eta) = \max_{\mu \in \ML} \frac{\I(\mu, \eta) }{\ell_X(\mu)}.
\end{equation}
Suppose now that the surface $\Sigma$ is closed. Following Bonahon \cite{Bonahon88} we can embed Teichm\"uller space $\CT$ as a subset of $\CC$ in such a way that for every $Y\in\CT$ and for every curve $\gamma$ we have
$$\I(\gamma,Y)=\ell_Y(\gamma).$$
We thus get from a result of Thurston \cite{Thurston} that
\[
d_{\Thu}(X,Y) = \log D_X(Y). 
\]
Since the function $ D_X \from \CC \to \BR_+$ is positive, homogeneous and continuous,
we can apply the Main Theorem for $\alpha=Y$, $f= D_X$. Letting $L= e^R$, we obtain:

\begin{sat}\label{Thurston}
Let $\Sigma$ be a closed surface of genus $g\ge 2$ an let $X,Y \in \CT(\Sigma)$ be two points in \Teich space. Then, 
\begin{equation*}
\lim_{R\to\infty}\frac{\# \big\{\phi\in\Map(\Sigma)  \st 
  d_{\Thu}\big(X,\phi(Y)\big)\le R\big\}}{e^{(6g-6)R}}
    =\frac{m(D_X) \, m(Y)}{\m_\Sigma},
\end{equation*}
where $d_{\Thu}$ is the Thurston metric on \Teich space and $D_X$ is as in
\eqref{DX}. \qed. 
\end{sat}

Erlandsson and the second author of this note have proved that Theorem \ref{Thurston} remains also true when the involved surface is not closed. This will appear in \cite{ES2}.

\begin{bem}
Unlike the \Teich metric, the Thurston metric $d_{\Thu}$ is not symmetric. This is reflected in the limiting constant of Theorem \ref{Thurston}. 
On the other hand, replacing $d_{\Thu}\big(X,\phi(Y)\big)\le R$ by 
$d_{\Thu}(\phi(Y),X)\le R$ will not make a difference because 
$$
\big\{\phi  \in\Map(\Sigma)  \st  d_{\Thu}\big(\phi(Y),X\big)\le R\big\}
= \big\{\phi \in\Map(\Sigma)  \st d_{\Thu}\big(Y,\phi^{-1}(X)\big)\le R\big\}.
$$
\end{bem}

\subsection*{Convergence of measures}
We now comment briefly on the proof of the main Theorem. Consider the space 
$\Measure(\CC)$ of Radon measures on $\CC$, endowing it in turn with the 
weak-*-topology induced by 
the algebra of continuous functions with compact support $C^0_{\rm c}(\CC)$. 
In plain language, this means that, for $\nu_n, \nu \in \Measure(\CC)$, 
\[
\nu_n \to \nu \quad\text{in $\Measure(\CC)$} 
\quad\Longleftrightarrow\quad
\int_\CC f\nu_n \to \int_\CC f\nu \quad \text{ for every $f\in C^0_c(\CC)$}. 
\]
Denoting by $\delta_\alpha\in\Measure(\CC)$ the Dirac measure centered at a compactly supported current $\alpha \in \CC_c(\bSigma)$, we will derive the Main Theorem from the convergence of the measures 
\begin{equation}
\nu_{\alpha}^L=\frac 1{L^{6g+2n-6}}\sum_{f\in\Map(\Sigma)}\delta_{\frac 1Lf(\alpha)} 
   \in \Measure(\CC)
\end{equation}  
as $L$ tends to $\infty$. We prove:

\begin{sat}\label{key}
With notation as in the Main Theorem we have 
\[
\lim_{L\to\infty}\nu_\alpha^L= \frac{m(\alpha)}{\m_\Sigma} \cdot \mu_{\Thu}.
\]
\end{sat}

\begin{bem}
A version of this theorem, where $\alpha$ is assumed to be a curve appears in \cite{ES},
however, the constant there is not explicit. 
\end{bem}

The strategy to prove Theorem \ref{key} is as follows. 
First we show that the family $(\nu_\alpha^L)_L$ is precompact (Lemma \ref{lem-precompact}), meaning that every
sequence of $L_n\to\infty$ has a subsequence $(L_{n_i})$ such that the measures 
$\nu_\alpha^{L_{n_i}}$ converge to some measure $\nu$. 
We prove next that any such limit $\nu$ is supported by the space $\CM\CL$ of measured laminations (Lemma \ref{lem-ml}), and that it gives measure $0$ to the set of non-filling laminations (Lemma \ref{lem-full}). It then follows from Lindenstrauss-Mirzakhani's classification of locally finite mapping class group invariant measures on $\CM\CL$ 
\cite{LM, Ham} that the limit $\nu$ is a multiple $c\cdot\mu_{\Thu}$ of the Thurston measure (Lemma \ref{lem-thurston}). 
All that remains is to show that $c= \frac{m(\alpha)}{\m_\Sigma}$. We deduce that this is
the case, using a slightly improved version of Equation \eqref{Maryam}: in 
Theorem \ref{Curve-Counting} we show that the constant $n_\gamma$ in Equation \eqref{Maryam} equals $m(\gamma)$. Mirzakhani's proof of \eqref{Maryam} relies on 
existence of the limit
\[
\lim_{L\to \infty} \frac{\vol_{\WP} \big(\big\{ X \in \CT  \ \big|\,  \ell_X(\gamma) \leq L \big\}\big) }{L^{6g+2n-6}}.
\]
We make use of results of Bonahon-S\"ozen \cite{BS} and Papadopoulos \cite{Papa} to reprove the existence of this limit and further compute its value (Theorem \ref{Volume}).

\subsection*{Further comments} We add a few comments that the reader may find interesting. 

\subsection*{1.} Let $\Gamma$ be a finite index subgroup of $\Map(\Sigma)$. 
The results of this paper also hold for $\Gamma$ up to dividing the constants
by the index $[\Map(\Sigma) : \Gamma]$. For instance, if $\alpha$ and $f$ are as in 
the Main Theorem, then 
\[
  \lim_{L\to \infty} 
  \frac{ \# \big\{ \phi \in \Gamma \st f\big(\phi(\alpha)\big) \leq L \big\}}{L^{6g+2n-6}}= 
  \frac 1{[\Map(\Sigma) : \Gamma]} \cdot \frac{m(\alpha) \, m(f)}{ \m_g}. 
\]
This is because everything we use, specifically the results in \cite{M}, also holds for $\Gamma$. 

\subsection*{2.} One should definitively point out that the results in this paper do not apply if the surface $\Sigma$ is not orientable. In fact, already Mirzakhani's Equation \eqref{Maryam} fails in that case, \cite{Orientable1,Orientable2}.

\subsection*{3.} The arguments in this paper show that not only does Mirzakhani's \eqref{Maryam} imply Theorem \ref{key} and thus the Main Theorem, but that also the implications can be reversed. It would be interesting to see if the lattice counting theorem with respect to the Teichm\"uller metric \cite{ABEM} is also equivalent to these results.

\subsection*{Acknowledgements} 
The work of Maryam Mirzakhani provides the foundation for the results in this paper 
and in general she has been an inspiration for us. We dedicate this paper, 
which hopefully she would have found amusing, to her memory. We also thank the referee
for helpful comments. 

\section{Currents} \label{Sec:Prelim}
In this section, we recall a few facts on currents. See \cite{Bonahon86,Bonahon88} and \cite{Javi-Chris} for a thorough treatment.

\subsection*{Definitions and notation}
Continuing with the same notation as in the introduction, let $\bSigma$ be a compact surface with negative Euler-characteristic. Let $g$ be the genus and $n$ the number of boundary components of $\bSigma$. 

We endow $\bSigma$ with a fixed but otherwise arbitrary hyperbolic metric with respect 
to which the boundary is geodesic and let 
$PT_{\rm rec}\bSigma$ (resp. $T^1_{\rm rec}\bSigma$) be the subset of the projective 
tangent bundle $PT\bSigma$ (resp. unit tangent bundle $T^1\bSigma$) consisting of lines (resp. vectors) tangent to bi-infinite geodesics. By construction, $PT_{\rm rec}\bSigma$ has a $1$--dimensional foliation whose leaves are the lifts of geodesics. 

A {\em geodesic current} is a Radon transverse measure to the geodesic foliation of $PT_{\rm rec}\bSigma$. Equivalently, a geodesic current is a $\pi_1(\bSigma)$--invariant Radon measure on the space of unoriented geodesics in the universal cover of $\bSigma$. Also,  geodesic currents are in one-to-one correspondence with Radon measures on $T^1_{\rm rec}\bSigma$ which are invariant under both the geodesic flow and the geodesic flip. Recall that a Borel measure is Radon if it is locally finite and inner regular. We denote the space of all geodesic currents, endowed with the weak-*-topology, by $\CC(\bSigma)$. 

It follows from the very definition of current that we can consider (unoriented) periodic orbits of the geodesic flow, that is closed geodesics, as geodesic currents. In fact, the set $\BR_+\CS$ of weighted closed geodesics is dense in $\CC(\bSigma)$. A measured lamination is a lamination of $\Sigma=\bSigma\setminus\D\bSigma$ endowed with a transverse measure of full support. As such, a measured lamination is also a current. Actually, the space $\ML(\Sigma)$ of measured laminations is nothing but the closure in $\CC(\bSigma)$ of the set of all weighted simple curves other than the boundary. Thurston proved that $\CM\CL(\Sigma)$ is, with the induced topology, homeomorphic to $\BR^{6g+2n-6}$. 

\subsection*{Mapping class group action}
The space $\CC(\bSigma)$ of currents is independent of the chosen metric on $\bSigma$ in the following sense: Any homeomorphism $\bSigma\to\bSigma'$ between hyperbolic surfaces with geodesic boundary induces a homeomorphism $PT_{\rm rec}\bSigma\to PT_{\rm rec}\bSigma'$ mapping one geodesic foliation to the other. Then the map $\CC(\bSigma)\to\CC(\bSigma')$ between the spaces of currents induced by the foliation preserving homeomorphism $PT_{\rm rec}\bSigma\to PT_{\rm rec}\bSigma'$ is also a homeomorphism. Note in particular that the mapping class group 
$$\Map(\Sigma)=\Homeo_+(\Sigma)/\Homeo_0(\Sigma)=\Homeo_+(\bSigma)/\Homeo_0(\bSigma)$$ 
of the interior $\Sigma=\bSigma\setminus\D\bSigma$ acts on $\CC(\bSigma)$ by homeomorphisms.

The space $\CM\CL(\Sigma)$ is not only homeomorphic to euclidean space, but has also a compatible mapping class group invariant integral PL-manifold structure. Here, integral means that the change of charts are given by linear transformations with integral coefficients. 

\subsection*{Compactly supported currents}
We will be mostly interested in currents in the set 
\[
\CC_0(\bSigma)
  =\big\{\lambda\in\CC(\bSigma)\text{ with }\lambda(\D\bSigma)=0\big\},
\]
that is in those currents which do not assign any weight to the boundary components of 
$\bSigma$. Also, for $K\subset\Sigma$ compact let 
\[
\CC_K(\bSigma)
=\big\{\lambda\in\CC(\bSigma)\text{ with }\supp(\lambda)\subset\CG_K(\bSigma)\big\}
\]
be the set of currents supported by the set of $\CG_K(\bSigma)$ of geodesics in the universal cover of $\bSigma$ whose projection to the base is contained in $K$. Finally set
\[
\CC_c(\bSigma)=\bigcup_{K\subset\Sigma\text{ compact}}\CC_K(\bSigma)
\]
to be the set of currents supported by some compact set. We refer to the element of $\CC_c(\bSigma)$ as {\em compactly supported currents.} We make a 
few comments on compactly supported currents:
\begin{itemize}
\item While $\CC(\bSigma)$ and $\CC_K(\bSigma)$ are projectively compact, neither $\CC_0(\bSigma)$ not $\CC_c(\bSigma)$ are if $\D\bSigma\neq\emptyset$. This is basically the reason why one needs to be more careful while working with currents on open surfaces.
\item For every $K\subset\Sigma$ compact there is another compact set 
$K'\subset\Sigma$ with 
\[
\Map(\Sigma)\cdot\CC_K(\bSigma)\subset\CC_{K'}(\bSigma).
\] 
In particular, both $\CC_0(\bSigma)$ and $\CC_c(\bSigma)$ are invariant under the mapping class group.
\item Let $X$ be a hyperbolic metric on the interior $\Sigma$ of $\bSigma$. The length function $\ell_X$ on the set of curves $\CS$ extends continuously to a positive homogenous function on the space $\CC_c(\bSigma)$ of compactly supported currents. In fact, many other length functions extend as well \cite{Hugo}.
\item There is some compact set $K\subset\Sigma$ with 
$\CM\CL(\Sigma)\subset\CC_K(\bSigma)$. In other words, measured laminations are 
compactly supported currents. 
\end{itemize}

\subsection*{Intersection pairing}
In \cite{Bonahon86} Bonahon introduced the map
\begin{equation}\label{eq-intersection}
\I:\CC(\bSigma)\times\CC(\bSigma)\to\BR_+, 
  \qquad \I(\lambda,\mu)=(\lambda\otimes\mu)(DPT\bSigma),
\end{equation}
where $DPT\bSigma$ is the bundle over $\bSigma$ whose fibre over $p$ is the set
$$
DPT\bSigma_p=\big\{(v,w)\in PT_p\bSigma\times PT_p\bSigma\text{ with }v\neq w\big\}
$$
consists of pairs of distinct points in the projective space of the tangent space $T_p\bSigma$. Bonahon proved that $\I(\param,\param)$ is continuous and homogenous, 
where homogenous means that
$$\I(t\cdot\lambda,s\cdot\mu)=s\cdot t\cdot\I(\lambda,\mu)$$
for all $\lambda,\mu\in\CC(\bSigma)$ and $t,s\in\BR_+$. Moreover, it follows directly from 
the definition that for any two closed primitive geodesics $\alpha,\beta$ in $\bSigma$ 
the quantity $\I(\alpha,\beta)$ is nothing other than the minimal number of transversal 
intersections of curves homotopic to $\alpha$ and $\beta$ and in general position. 
The map $\I$ is called the {\em intersection pairing}.

While $\I(\alpha,\param)=0$ whenever $\alpha$ is a boundary component of $\bSigma$, 
we have that the intersection pairing is non-degenerate on the set $\CC_0(\bSigma)$ of 
currents supported by the interior of $\bSigma$. In fact, measured laminations are 
characterized to be those currents in $\CC_0(\bSigma)$ with vanishing self-intersection 
number:
$$
\CM\CL(\Sigma)=\big\{\lambda\in\CC_0(\bSigma)\text{ with }\I(\lambda,\lambda)=0\big\}.
$$
On the other extremum, a current $\lambda\in\CC_0(\bSigma)$ is called {\em filling} if 
$\I(\lambda,\mu)>0$ for every non-zero current $\mu\in\CC_0(\bSigma)$. Equivalently, 
a current is filling if its support meets transversally every bi-infinite geodesic other than 
the boundary components.

We list a few facts on $\I(\param,\param)$:
\begin{itemize}
\item If $\phi \from \bSigma\to\bSigma'$ is a homeomorphism between hyperbolic surfaces,
then the induced homeomorphism $\phi^\star \from \CC(\bSigma)\to\CC(\bSigma')$ 
commutes with the scaling action of $\BR_+$ and with the intersection pairing. 
In particular, the intersection paring is invariant under the mapping class group action 
$\Map(\bSigma)\actson\CC(\bSigma)$.
\item If $K\subset\Sigma$ is compact and if $\lambda\in\CC_0(\bSigma)$ is a filling current, then the set 
\[
\big\{\mu\in\CC_K(\Sigma)\st \I(\lambda,\mu)\le 1\big\}
\] 
is compact in $\CC_K(\Sigma)$.
\item For any two filling currents $\mu,\lambda\in\CC_c(\Sigma)$ and for every compact set $K\subset\Sigma$, there is some $C$ with
$$C^{-1}\cdot\I(\mu,\eta)\le\I(\lambda,\eta)\le C\cdot\I(\mu,\eta)$$
for every $\eta\in\CC_K(\Sigma)$.
\item Suppose that $\mu\in\CC_c(\Sigma)$ is a filling current and that $X\in\CT(\Sigma)$ is a complete hyperbolic structure on the interior $\Sigma$ of $\bSigma$. For every compact set $K\subset\Sigma$, there is some $C$ with
$$C^{-1}\cdot\I(\mu,\eta)\le \ell_X(\eta)\le C\cdot\I(\mu,\eta)$$
for every $\eta\in\CC_K(\Sigma)$.
\end{itemize}

Before moving on, suppose for a moment that $\bSigma$ is closed, that is that it has empty boundary. Then, for every point $X\in\CT(\bSigma)$ in Teichm\"uller space, there is a current, again denoted by $X$ and called the {\em Liouville current}, with 
$$\I(X,\gamma)=\ell_X(\gamma)$$ 
for every closed curve $\gamma$. The Liouville current $X$ is a filling current. 

\subsection*{A comment on the continuity of the $\I(\param,\param)$}
In the literature, currents are often either discussed on the setting of hyperbolic groups or for closed surfaces. Considering currents on $\bSigma$ is consistent with the former point of view. What is specific about surfaces is the existence of the intersection pairing -- continuity being the only non-trivial result. We explain briefly how the continuity of the intersection pairing follows from the result in the closed case. Let $\hat\Sigma$ be a closed hyperbolic surface for which there is an isometric embedding $\bSigma\hookrightarrow\hat\Sigma$. For instance, take $\hat\Sigma$ to be the double of $\bSigma$. The inclusion of $\bSigma$ into $\hat\Sigma$ induces a foliation preserving embedding $PT_{\rm rec}\bSigma\to PT_{\rm rec}\hat\Sigma$. This embedding induces an embedding $\CC(\bSigma)\to\CC(\hat\Sigma)$. By definition, the intersection pairing on $\CC(\bSigma)$ is nothing other than the restriction of the intersection pairing on $\CC(\hat\Sigma)$. Continuity of the latter then implies continuity of the former.

\section{Weil-Petersson Volume} \label{WP}
In this section, we compute $n_\gamma$ in
Equation \eqref{Maryam} proving the following version of the Mirzakhani's
theorem \cite[Theorem 1.1]{M}. 

\begin{sat} \label{Curve-Counting}
Let $\bSigma$ be a compact surface of genus $g$ with $n$ boundary components and let $X$ be a complete finite volume hyperbolic structure on its interior $\Sigma=\bSigma\setminus\D\bSigma$. Then 
\begin{equation*}
  \lim_{L\to \infty} 
  \frac{\# \big\{ \phi \in \Map(\Sigma) \st \ell_X(\phi(\gamma)) \leq L \big\}}{L^{6g+2n-6}}
   = \frac{m(X) \, m(\gamma)}{\m_\Sigma}
\end{equation*}
for every filling closed curve $\gamma\subset\Sigma$.
\end{sat}

We need some preparation before coming to the proof. As we mentioned earlier, $\ML$ 
is naturally a PL-manifold. Moreover, it is endowed with a mapping class group invariant symplectic structure, namely Thurston's intersection pairing 
\cite{Penner-Harer}. Actually, this symplectic structure is compatible with the Weil-Peterson form on $\CT$. More
precisely, let $\mu$ be a geodesic lamination where the complementary components 
are ideal triangles and let $\ML(\mu)$ be the open and dense subset of $\CM\CL$ 
consisting of measured laminations that are transverse to $\mu$. 
Note that $\mu$ itself may not be a measured lamination. In fact, in case $\bSigma$
has boundaries, we take (for simplicity) $\mu$ to be a union of finitely many bi-infinite 
geodesics. Thurston has showed \cite{Thurston} that there is a global parametrization
\[
\Phi_\mu \from \CT \to \ML(\mu). 
\]
sending a hyperbolic structure $X$ to a measured lamination associated to the 
measured foliation $F$ transverse to $\mu$ such that the $F$--measure of a sub-arc 
$\omega$ of $\mu$ is the hyperbolic length of $\omega$ in $X$. The map $\Phi_\mu$ 
is relevant in the present setting becauseit was shown by Bonahon and S\"ozen proved 
in the closed case \cite{BS} and by Papadopoulos and Penner in the case where 
$\mu$ is finite \cite{PaP} that it is a symplectomorphism when the source is endowed 
with the Weil-Petersson form and the target with Thurston's intersection pairing.

Now, the symplectic structure on $\ML$ induces a mapping class 
group invariant measure 
in the Lebesgue class. This measure is the so-called {\em Thurston measure} $\mu_{\Thu}$. 
It is an infinite but locally finite measure, positive on non-empty open sets.
Further we have
$$\mu_{\Thu}(L\cdot U)=L^{6g-6}\cdot\mu_{\Thu}(U)$$ 
for all $U\subset\ML(\Sigma)$ and $L>0$. This implies that $\mu_{\Thu}(A)=0$ if 
$A\cap L\cdot A=\emptyset$ for all $L>0$. In particular,
$$\mu_{\Thu}\big(\big\{\lambda\in\CM\CL(\Sigma) \st  f(\lambda)=L\big\}\big)=0$$
for any positive continuous homogenous function $f$ on the space of currents. Finally, note that $\ML(\mu)$ has full measure.  

\begin{bem}
It is due to Masur \cite{Masur} that, up to scaling, the Thurston measure is the only $\Map(\Sigma)$
invariant measure on $\ML$ in the Lebesgue class. However, several natural normalization
are possible and it is not known if they agree. For example, one could normalize the 
measure to be equal to the scaling limit 
$$\lim_{L\to\infty}\frac 1{L^{6g-6}}\sum_{\gamma\in\CM\CL_\BZ}\delta_{\frac 1L\gamma}$$
of the counting measure on integral multi-curves. For a discussion of how 
the measure defined by the symplectic structure (which is what we use in 
this paper) is related to the measure defined by the scaling limit (used by Mirzakhani 
in \cite{M}) see \cite{Ivan-Leoni} 
\end{bem}

The main idea in \cite{M} is to relate the counting problem in \eqref{Maryam} to 
the Weil-Petersson volume of certain sets in $\CT$. More precisely, for a filling curve
$\gamma$ and $L>0$, Mirzakhani considered the sets
\[
B_\CT(\gamma,L) = \Big\{ X \in \CT  \ \Big|\,  \ell_X(\gamma) \leq L \Big\}
\]
and showed \cite[Theorem 8.1]{M} that
\begin{equation} \label{Volume-Maryam}
\lim_{L\to \infty} \frac{\vol_{\WP} \big( B_\CT (\gamma, L) \big) }{L^{6g-6}}
\end{equation} 
exists. We reprove this theorem, further giving a precise value for the limit. 
First, we recall the follow lemma from \cite[Lemma 4.9]{Papa}.

\begin{lem}[Papadopoulos] \label{Papa}
Let $X_n$ be a sequence of points in $\CT$ converging to a lamination $\lambda$
in the Thurston boundary. Then, for every curve $\alpha$, there exists a constant $C$ 
such that for every $n$, we have 
\[
\I\big(\Phi_\mu(X_n), \alpha\big)\leq \ell_{\alpha}(X_n) 
    \leq  \I \big(\Phi_\mu(X_n), \alpha\big)+C.
\]
\end{lem}

\begin{sat}\label{Volume}
We have 
\[
\lim_{L \to \infty} \frac{\vol_{\WP} \big( B_\CT (\gamma, L) \big)}{L^{6g-6}} = m(\gamma)
\]
for any filling curve $\gamma$.
\end{sat}

\begin{proof}
Recall that the map $\Phi_\mu$ considered earlier is a symplectomorphism and thus volume preserving. Thus, to compute the Weil-Petersson volume of $B_\CT (\gamma, L)$, we can instead find out the Thurston volume of $B_L = \Phi_\mu \big(  B_\CT (\gamma, L) \big)$:
$$\vol_{\WP}(B_\CT (\gamma, L))=\mu_{\Thu}(B_L).$$
Note that it follows from the convexity of lengths with respect to the shearing coordinates 
\cite{BBFS, Theret} that $B_L$ is a closed convex set in $\ML(\mu)$. 

Since Thurston's volume form
is compatible with the linear structure in $\ML$, we have 
\[
\mu_{\Thu} (B_L) = L^{6g-6} \mu_{\Thu} \left (\frac 1L B_L \right)
\]
where 
\[
\frac 1L B_L =  \left\{ \frac 1L \lambda  \ \Big|\,  \lambda \in B_L \right\}.
\]

Also, it follows immediately from Lemma~\ref{Papa} that, for sequences $X_n \in \CT$ and 
$L_n>0$ and a measure lamination 
$\lambda \in \ML$, if $\frac{1}{L_n}X_n \to \lambda \in\CM\CL$ as a sequence of geodesic 
currents in $\CC(S)$, then 
\[
\frac{1}{L_n} \Phi_\mu(X_n) \to \lambda.
\]
This implies that the family of sets $\frac{1}{L} B_L$ converge point-wise to the set
\[
B_\ML(\gamma,1) = \Big\{ \lambda \in \ML(\mu)   \,\Big|\,  \I(\gamma, \lambda) \leq 1 \Big\}.
\]
That is, we have a family of closed convex sets that converge 
point-wise to a closed set. Then their volume converges as well. Altogether we get
$$\frac 1{L^{6g-6}}\vol_{\WP}(B_\CT (\gamma, L))=\frac 1{L^{6g-6}}\mu_{\Thu}(B_L)=\mu_{\Thu} \left (\frac 1L B_L \right)\longrightarrow B_\ML(\gamma,1).$$
And $m(\gamma)$ is nothing but the Thurston volume of $B_\ML(\gamma, 1)$ because $\ML(\mu)$ has full measure. 
\end{proof}

Mirzakhani derived Equation \eqref{Maryam} from the existence
of the limit \eqref{Volume-Maryam}. To obtain Theorem~\ref{Curve-Counting}
we need to show:

\begin{sat} For any filling curve $\gamma$, 
\begin{equation} \label{above}
\# \big\{ \phi \in \Map(\Sigma) \st \ell_X(\phi(\gamma)) \leq L \big\}
\sim \vol_{\WP} \big( B_\CT(\gamma, L) \big)\cdot \frac{m(X)}{\m_\Sigma},
\end{equation} 
where $\sim$ means that the ratio tends to $1$ when $L\to\infty$. 
\end{sat} 

\begin{proof}
This is essentially proven as part of proof of Theorem 1.1 in \cite{M} 
(see \cite[Section 9.4]{M}). Let 
$P= \{ \alpha_1, \dots, \alpha_{3g-3} \} $ be a pants decomposition 
of $\Sigma$ and consider the the associated Fenchel-Nielsen coordinates for 
\Teich space, namely 
\[
X \to \Big( \ell_1(X), \tau_1(X), \dots, \ell_{3g-3}(X), \tau_{3g-3}(X) \Big) . 
\]
where $\ell_i(X)$ is the length of $\alpha_i$ at $X$ and $\tau_i(X)$ is the amount of 
twisting around the curve $\alpha_i$. For ${\bf m} = (m_1, \dots, m_{3g-3}) \in \BZ^{3g-3}$ define 
$C_P^{\bf m}$ to be the cone in this coordinate consisting of all points $X \in \CT$ where 
\[
m_i \cdot \ell_i(X) \leq \tau_i(X) \leq (m_i+1) \cdot \ell_i(X).
\] 
Now, \cite[Equation 9.4]{M} states that 
\begin{equation} \label{cone}
\# \big\{  \phi \in \Map(\Sigma) \st 
  \phi(X) \in  C_P^{\bf m} \cap  B_\CT(\gamma, L)  \big\}
   \sim \vol_{\WP} \big( C_P^{\bf m} \cap  B_\CT(\gamma, L)  \big)
      \cdot \frac{m(X)}{\m_\Sigma}. 
\end{equation} 
That is, the theorem holds for every cone. However, we need infinitely many such cones 
to cover $B_\CT(\gamma, L)$ and we need to be a bit careful. 

As one applies larger and larger powers of the Dehn twist around the curve $\alpha_i$, 
the length of $\gamma$ (which has to intersect $\alpha_i$ because $\gamma$ is filling)
eventually becomes larger than $\tau_i(X) \cdot \ell_i(X)$. Therefore, there is a
constant $E_0=E_0(\gamma, X)$ depending on $\gamma$ and $X$ so that, for $L$ 
large enough, 
\[
\phi(X) \in  C_P^{\bf m} \cap  B_\CT(\gamma, L)
\quad\Longrightarrow \quad 
\ell_{\alpha_i}(\phi (X)) \leq E_0 \cdot \frac{L}{m_i}. 
\]
Using the fact that $\WP$--volume form can be written as
$d\ell_1 \cdots d\ell_{3g-3} d\tau_1 \cdots d\tau_{3g-3}$ and
setting $E_2 = E_0^{6g-6}$, we also have 
\begin{equation} \label{L/m-vol}
 \vol_{\WP} \big( C_P^{\bf m} \cap  B_\CT(\gamma, L)  \big) \leq
E_2 \frac{L^{6g-6}}{m_1^2 m_2^2 \cdots m_{3g-3}^2}. 
\end{equation}

Recall \cite[Lemma 5.1]{M} which states that, for some $E_1$ depending on $E_0$,
\[
\# \Big\{ \phi(P) \st
    \ell_{\phi(P)}(X) \leq L, \ \forall i \ \ell_{\phi(\alpha_i)}(X) \leq \frac{L}{m_i} \Big\}
    \leq E_1 \frac{L^{6g-6}}{m_1^2 m_2^2 \cdots m_{3g-3}^2}. 
\]
This can be restated as 
\begin{equation} \label{L/m-phi}
\# \big\{  \phi \in \Map(\Sigma) \st 
\phi(X) \in  C_P^{\bf m} \cap  B_\CT(\gamma, L)  \big\} \leq
E_1 \frac{L^{6g-6}}{m_1^2 m_2^2 \cdots m_{3g-3}^2},
\end{equation} 
because, for every $\phi$ in the first set, a composition $\phi^{-1}$ with the correct number of
Dehn twists around the curves $\alpha_i$ is in the second set and this gives a bijection between the two sets. But 
\[
\sum_{{\bf m} \in \BN^{3g-3}} \frac{1}{m_1^2} \times \frac{1}{m_2^2} \times \dots \times
\frac{1}{m_{3g-3}^2} \leq \infty. 
\] 
Therefore, for every $\epsilon$, there exists a finite collection of cones $C_P^{\bf m}$
so that the proportion of the contribution to both sides of 
Equation~\eqref{above} from all the other cones is less than $\epsilon$, independent 
of the value of $L$. Hence, Equation~\eqref{cone} implies Equation~\eqref{above}. 
\end{proof}

Armed with Theorem \ref{Curve-Counting}, we can obtain a special case of 
the Main Theorem where the function $f$ is given by intersection number and where, 
more crucially, the current is assumed to be a filling curve. This simpler statement
will be used later in the proof of the Main Theorem. 

\begin{kor}\label{kor-combined}
For every filling curve $\gamma$ in $\Sigma$ and for every filling current $\alpha\in\CC$ 
we have
$$
\lim_{L\to\infty} 
  \frac{\# \big\{\phi\in\Map(\Sigma) \st \I(\alpha,\phi(\gamma))\le L\big\} }{L^{6g-6}}=\frac{m(\gamma) m(\alpha) }{\m_g}.
$$
\end{kor}

\begin{proof} 
Recall, from \cite[Corollary 4.4]{ES}, that 
$$
\lim_{L\to\infty}\frac{\# \big\{\phi\in\Map(\Sigma) \st \I(\alpha,\phi(\gamma_0))\le L\big\}}
  {\#\big\{\phi\in\Map(\Sigma) \st \I(X,\phi(\gamma))\le L\big\}}
  =\frac{m(\alpha)}{m(X)}
$$
for any $X \in \CT$. The claim follows from this and Theorem \ref{Curve-Counting}. 
\end{proof}

\section{Limit of measures}
In this section we prove Theorem \ref{key}. Recall that $\Measure(\CC)$ is the space 
of Radon measures on $\CC=\CC(\bSigma)$, endowed with the weak-*-topology. Also, fix a filling compactly supported current $\alpha\in\CC_c(\bSigma)$ and recall the definition of the measures 
$$
\nu_{\alpha}^L=\frac 1{L^{6g+2n-6}}\sum_{\phi \in\Map(\Sigma)}
  \delta_{\frac 1L\phi(\alpha)}\in \Measure(\CC)
$$
from the introduction. The first step in the proof of Theorem \ref{key} is to show that the family of measures $\nu_\alpha^L$ is precompact, meaning that any sequence $L_n\to\infty$ has a subsequence $(L_{n_i})$ such that $(\nu_\alpha^{L_{n_i}})_i$ converges. 

\begin{lem}\label{lem-precompact}
The family $(\nu_\alpha^L)_L$ is precompact in $\Measure(\CC)$ and every accumulation point is locally finite and positive.
\end{lem}

\begin{proof}
First of all, we fix some finite area hyperbolic structure $X$ on $\Sigma$. We also fix 
a filling curve $\gamma_0$ and some $C\ge 1$ with
\begin{equation}\label{eq1}
C^{-1}\cdot\I(\lambda,\gamma_0)\le\I(\lambda,\alpha)\le C\cdot\I(\lambda,\gamma_0)
\end{equation}
for every current $\lambda\in\CC_c(\bSigma)$. 

Note now that there is some compact set $K\subset\Sigma=\bSigma\setminus\D\bSigma$ with $\phi(\alpha)\in K$ for all $\phi\in\Map(\Sigma)$. It follows that $\nu_\alpha^L$ is supported on $\CC_K(\bSigma)$ for all $K$. This is a key fact because $\{\lambda\in\CC_K(\bSigma) \st \ell_X(\lambda)\le T\}$ is compact for all $T\ge 0$ and in the weak-*-topology on a compact space, the space of probability measures is compact. Therefore, precompactness of the family $(\nu_\alpha^L)_L$ follows once we prove that 
$$
 \limsup_{L\to\infty}\nu_\alpha^L\big(\{\lambda\in\CC_K(\bSigma) \st \ell_X(\lambda)\le T\}\big)<\infty
$$
for every $T\ge 0$. Let us compute:
\begin{align*}
\nu_\alpha^L\big( \{\lambda\in\CC_K(\bSigma)\st\ell_X(\lambda)\le T\} \big)
&=\frac 1{L^{6g+2n-6}}\#\left\{\phi\in\Map(\Sigma) 
     \, \Big| \, \ell_X\left(\frac 1L \phi(\alpha)\right)\le T\right\}\\
&=\frac 1{L^{6g+2n-6}}\#\big\{\phi\in\Map(\Sigma)
    \st \I(X,\phi(\alpha))\le T\cdot L\big\}\\
&=\frac 1{L^{6g+2n-6}}\#\big\{\phi\in\Map(\Sigma)
    \st\I(\phi^{-1}(X),\alpha)\le T\cdot L\big\}\\
&\le\frac 1{L^{6g+2n-6}}\# \big\{\phi\in\Map(\Sigma)
    \st\I(\phi^{-1}(X),\gamma_0)\le C\cdot T\cdot L\big\}\\
&=\frac 1{L^{6g+2n-6}}\#\big\{\phi\in\Map(\Sigma)
    \st\I(X,\phi(\gamma_0))\le C\cdot T\cdot L\big\}\\
&=\frac 1{L^{6g+2n-6}}\#\big\{\phi\in\Map(\Sigma)
    \st\ell_X(\phi(\gamma_0))\le C\cdot T\cdot L\big\}.
\end{align*}
Now, from Theorem \ref{Curve-Counting} we get that, when $L$ grows, the last quantity approaches the quantity 
\[
(C\cdot T)^{6g+2n-6} \frac{m(\gamma_0) m(X)}{\m_\Sigma}.
\] 
In particular, the function 
$L\to \nu_\alpha^L\big(\big\{\lambda\in\CC_K(\bSigma)\st\ell_X(\lambda)\le T\big\}\big)$
is bounded, as we needed to show. 

Note that we have also already proved that every accumulation point of $(\nu_\alpha^L)_L$ is locally finite. To prove that any such accumulation point is positive, it suffices to prove conversely that
$$
\liminf_{L\to\infty}\nu_\alpha^L
  \big(\big\{\lambda\in\CC_K(\bSigma) \st \ell_X(\lambda)\le 1\big\}\big)>0.
$$
A completely analogous calculation as above, using this time the second inequality of \eqref{eq1} implies that 
$$
\nu_\alpha^L(\big\{\lambda\in\CC \st \ell_X(\lambda)\le 1\big\})\ge 
\frac {\# \big\{\phi\in\Map(\Sigma) \st \ell_X(\phi(\gamma_0))\le K^{-1}\cdot L\big\}}{L^{6g+2n-6}}.
$$
The claim follows, once again, from Theorem \ref{Curve-Counting}.
\end{proof}

Lemma \ref{lem-precompact} asserts that $(\nu_\alpha^L)_L$ has accumulation points. Over the next few lemmas we will study such points. 

\begin{lem}\label{lem-ml}
Suppose that $L_n\to\infty$ is a sequence such that the limit 
\[
\nu=\lim_{n\to\infty}\nu_\alpha^{L_n}
\] 
exists in $\Measure(\CC)$. Then $\nu$ is supported by the subspace $\CM\CL\subset\CC$ of measured laminations.
\end{lem}
\begin{proof}
Fix again an auxiliary hyperbolic structure $X\in\CT$. In order to prove that $\nu$ is supported by $\CM\CL$ it suffices to prove that
\begin{equation}\label{eq-tackymusic}
\int_{\{\lambda\in\CC\,\vert\,\ell_X(\lambda)\le L\}}\I(\lambda,\lambda)\, \nu(\lambda)=0
\end{equation}
for all $L$. Computing the corresponding quantity for $\nu_\alpha^L$ we get
\begin{align*}
\int_{\{\lambda\in\CC\,\vert\,\ell_X(\lambda)\le L\}}\I(\lambda,\lambda)\,\nu_\alpha^L(\lambda)
&=\frac 1{L^{6g+2n-6}}\sum_{\phi\in\Map(\Sigma),\ \ell_X(\phi(\alpha))\le L}\I\left(\frac{\phi(\alpha)}L,\frac{\phi(\alpha)}L\right)\\
&=\frac 1{L^{6g+2n-6}}\sum_{\phi\in\Map(\Sigma),\ \ell_X(\phi(\alpha))\le L}\frac {\I(\alpha,\alpha)}{L^2}\\
&=\frac{\# \big\{\phi\in\Map(\Sigma) \st \ell_X(\phi(\alpha))\le L\big\}}{L^{6g+2n-6}}
   \cdot\frac{\I(\alpha,\alpha)}{L^2}\\
&=\nu_\alpha^L\big(\big\{\lambda\in\CC \st\ell_X(\lambda)\le T\big\}\big)
    \cdot\frac{\I(\alpha,\alpha)}{L^2}
\end{align*}
where the second equality holds because of the homogeneity and invariance of the intersection pairing. The third equality holds because all the summands are identical and the fourth is just the definition of $\nu_\alpha^L$. Recall that, in the proof of Lemma \ref{lem-precompact}, we have shown
$$\limsup_{L\to\infty}\nu_\alpha^L(\{\lambda\in\CC\st \ell_X(\lambda)\le T\})<\infty.$$
Plugging this in the previous computation we deduce that
$$\lim_{L\to\infty}\int_{\{\lambda\in\CC\,\vert\,\ell_X(\lambda)\le L\}}\I(\lambda,\lambda)\nu_\alpha^L(\lambda)=0,$$
from where \eqref{eq-tackymusic} follows.
\end{proof}

We know at this point that every accumulation point $\nu$ of the family $(\nu_\alpha^L)_L$ 
is supported by the space a measured laminations. We will prove below that any such 
$\nu$ is a multiple of the Thurston measure. It is known \cite[Proposition 4.1]{ES} 
- and this is going to be important in the next Lemma - 
that this is the case if the filling current $\alpha$ is a filling curve:

\begin{prop}\label{multicurve-thurston}
Let $\eta$ be a filling multicurve and suppose that $L_n\to\infty$ is a sequence such that the limit $\nu=\lim_{n\to\infty}\nu_\eta^{L_n}$ exists in $\Measure(\CC)$. Then $\nu=c\cdot\mu_{\Thu}$ for some $c>0$. \qed 
\end{prop}

For now, we prove that, for a general filling current, any accumulation point of $(\nu_\alpha^L)$ when $L\to\infty$ gives measure $0$ to the set of laminations which fail to intersect some curve.

\begin{lem}\label{lem-full}
Suppose that $L_n\to\infty$ is a sequence such that the limit $\nu=\lim_{n\to\infty}\nu_\alpha^{L_n}$ exists. Then we have
$$\nu\big(\big\{\lambda\in\CM\CL\st\I(\lambda,\gamma)=0\big\}\big)=0$$
for every simple curve $\gamma\subset\Sigma$.
\end{lem}

\begin{proof}
Start by fixing an auxiliary hyperbolic structure $X\in\CT(\Sigma)$ and some $C\ge 1$ with 
$$C^{-1}\cdot\I(X,\lambda)\le\I(\alpha,\lambda)\le C\cdot\I(X,\lambda)$$
for every current $\lambda\in\CC$. The claim of the lemma follows when we show that for all $T>0$ we have
\begin{equation}\label{eq-bigsupport}
\lim_{\epsilon\to 0}\limsup_{L\to\infty}\nu_\alpha^L
  \big(\big\{\lambda\in\CC\st\I(\lambda,\gamma)\le\epsilon T
     \text{ and }\ell_X(\lambda)\le T\big\}\big)=0.
\end{equation}
Before launching the proof of \eqref{eq-bigsupport}, choose a filling curve $\eta$ and note that, up to increasing $K$ just this time only, we can further assume that 
$$K^{-1}\cdot\I(\eta,\lambda)\le\I(\alpha,\lambda)\le K\cdot\I(\eta,\lambda)$$
holds again for every current $\lambda\in\CC$. 

Computing basically as above we get the following:
\begin{align*}
\nu_\alpha^L\big(\big\{\lambda&\in\CC\st\I(\lambda,\gamma)\le\epsilon T
  \text{ and }\ell_X(\lambda)\le T\big\}\big)\\
&=\frac 1{L^{6g+2n-6}}\# \big\{\phi\in\Map(\Sigma)\st \I(\phi(\alpha),\gamma)\le\epsilon TL,\    
    \ell_X(\phi(\alpha))\le TL\big\}\\
&\le\frac 1{L^{6g+2n-6}}\#\big\{\phi\in\Map(\Sigma)\st\I(\phi(\eta),\gamma)\le\epsilon TLK,\ 
   \ell_X(\phi(\eta))\le TLK\big\}\\
&=\nu_\eta^L(\{\lambda\in\CC\st\I(\lambda,\gamma)\le\epsilon TK\text{ and }\ell_X(\lambda)\le TK\}).
\end{align*}
From Proposition \ref{multicurve-thurston} we get then that any limit of $\nu_\eta^L$ is a multiple of the Thurston measure. In other words, there is $c$ with
\begin{align*}
\lim\nu_\alpha^L\big(\big\{\lambda&\in\CC\st\I(\lambda,\gamma)\le\epsilon T
   \text{ and }\ell_X(\lambda)\le T\big\}\big)\le \\
&\le c\cdot\mu_{\Thu}\big(\big\{\lambda\in\CC\st\I(\lambda,\gamma)\le\epsilon TK
   \text{ and }\ell_X(\lambda)\le TK\big\}\big)
\end{align*}
for all $L$. Since $\mu_{\Thu}$ gives vanishing measure to the set of those laminations which do not intersect $\gamma$ we get that this last quantity tends to $0$ as $\epsilon\to 0$. We have proved \eqref{eq-bigsupport} and thus Lemma \ref{lem-full}.
\end{proof}

We recall that Lindenstrauss and Mirzakhani \cite{LM} have 
classified all mapping class group invariant measures on $\ML$
(see also \cite{Ham}). 
The following characterisation of the Thurston measure is just a reformulation of 
Theorem 7.1 in their paper:

\begin{sat*}[Lindenstrauss-Mirzakhani]
Let $\nu$ be a locally finite $\Map(\Sigma)$-invariant measure on $\CM\CL(\Sigma)$ with
$$\nu\big(\big\{\lambda\in\CM\CL(\Sigma)\st\I(\gamma,\lambda)=0\big\}\big)=0$$
for every simple closed curve $\gamma\subset\Sigma$. Then $\nu$ is a multiple of the Thurston measure $\mu_{\Thu}$.
\end{sat*}

We are now ready to prove that Proposition \ref{multicurve-thurston} also holds for arbitrary filling currents $\alpha$. In other words, we prove that any accumulation point of $(\nu_\alpha^L)$ is a multiple of the Thurston measure:

\begin{lem}\label{lem-thurston}
Suppose that $L_n\to\infty$ is a sequence such that the limit $\nu=\lim_{n\to\infty}\nu_\alpha^{L_n}$ exists. Then $\nu=c\cdot\mu_{\Thu}$ for some $c>0$.
\end{lem}
\begin{proof}
Note that the limiting measure $\nu$ is positive and locally finite by Lemma \ref{lem-precompact}. Moreover, $\nu$ is supported by $\CM\CL$ by Lemma \ref{lem-ml}. Also, since all the measures $\nu_\alpha^{L_n}$ are $\Map(\Sigma)$-invariant we get that $\nu$ is $\Map(\Sigma)$-invariant as well. Finally, by Lemma \ref{lem-full} we have
$$\nu\big(\big\{\lambda\in\CM\CL\st\I(\lambda,\gamma)=0\big\}\big)=0$$
for every simple curve $\gamma\subset\Sigma$. The claim follows from the 
Lindenstrauss-Mirzakhani theorem.
\end{proof}

We are now ready to prove Theorem \ref{key}:

\begin{proof}[Proof of Theorem \ref{key}]
By Lemma \ref{lem-precompact} we know that the any sequence $L_n\to\infty$ has a subsequence $(L_{n_i})$ such that the limit 
$$\nu=\lim_{n\to\infty}\nu_\alpha^{L_n}$$ 
exists. Then, by Lemma \ref{lem-thurston} we know that $\nu=c\cdot\mu_{\Thu}$ for some constant $c=c(L_{n_i})$ depending on the sequence $(L_{n_i})$. To prove 
the theorem, it suffices to show that 
\begin{equation} \label{c}
c = \frac{m(\alpha)}{\m_g},
\end{equation}
which implies in particular that $c$ is independent of the sequence $L_{n_i}$. 
To that end fix a curve $\gamma$ and compute:
\begin{align*}
\nu\big(\big\{\lambda\in\CC\st\I(\lambda,\gamma)\le 1\big\}\big)
&=\lim_{i\to\infty}\nu_\alpha^{L_{n_i}}
  \big(\big\{\lambda\in\CC\st\I(\lambda,\gamma)\le 1\big\}\big)\\
&=\lim_{i\to\infty}\frac 1{L_{n_i}^{6g+2n-6}}
  \# \big\{\phi\in\Map(\Sigma)\st  \I(\phi(\alpha),\gamma)\le L_{n_i}\big\}\\
&=\lim_{i\to\infty}\frac 1{L_{n_i}^{6g+2n-6}}
  \# \big\{\phi\in\Map(\Sigma)\st \I(\alpha,\phi(\gamma))\le L_{n_i}\big\}\\
&=\frac{m(\gamma)m(\alpha)}{\m_g},
\end{align*}
where the last line holds by Corollary \ref{kor-combined}. Since $\nu = c\cdot\mu_{\Thu}$
we have
$$
\nu \big(\big\{\lambda\in\CC\st\I(\lambda,\gamma)\le 1\big\}\big)=c\cdot m(\gamma),
$$
meaning that 
\[
c\cdot m(\gamma) = \frac{m(\gamma)m(\alpha)}{\m_g}. 
\]
Equation \eqref{c} follows and we have proved Theorem \ref{key}.
\end{proof}

\section{Proof of the Main Theorem}

We restate the Main Theorem from the introduction. 

\begin{named}{Main Theorem}
Let $\bSigma$ be a compact surface of genus $g$ and $n$ boundary components, and let $\Sigma=\bSigma\setminus\D\Sigma$ be its interior. For a positive, continuous and homogeneous map $f \from \CC(\bSigma) \to \BR_+$ and a filling compactly supported current $\alpha \in \CC_c(\bSigma)$, we have
\[
  \lim_{L\to \infty} 
  \frac{ \# \big\{ \phi \in \Map(\Sigma) \st f\big(\phi(\alpha)\big) \leq L \big\}}{L^{6g+2n-6}}= 
  \frac{m(\alpha) \, m(f)}{\m_\Sigma}. 
\]
\end{named}

\begin{proof}
For $L>0$, consider again the measures
\[
\nu_{\alpha}^L=\frac 1{L^{6g+2n-6}}
  \sum_{\phi\in\Map(\Sigma)} \delta_{\frac 1L \phi(\alpha)},
\]
recall that they are supported by $\CC_K(\bSigma)$ for some $K\subset\Sigma$ compact, and note that for all $L$ one has
\begin{align*}
\#\big\{\phi\in\Map(\Sigma)\st  f(\phi(\alpha)) \le L\big\}
  &=\left(\sum_{\phi\in\Map(\Sigma)}\delta_{\phi(\alpha)}\right)
    \left(\big\{\eta \in \CC_K(\bSigma) \st  f(\eta)\le L \big\}\right)\\
&=\left(\sum_{\phi\in\Map(\Sigma)} \delta_{\frac 1L\phi(\alpha)}\right)
    \left(\big\{\eta\in\CC_K(\bSigma) \st  f(\eta)\le 1\big\}\right)\\
&=L^{6g+2n-6}\cdot \nu^L_{\alpha} \big(\big\{\eta\in\CC_K(\bSigma) \st  f(\eta)\le 1\big\}\big).
\end{align*}
By Theorem \ref{key}, the limit of $\nu_\alpha^L$ as $L$ goes to infinity is 
$\frac{m(\alpha)}{\m_g} \mu_{\Thu}$. We have
\begin{align*}
\lim_{L\to\infty}\frac  {\# \{\phi\in\Map(\Sigma) \st f(\phi(\alpha))\le L\}} {L^{6g+2n-6}}
 & = \frac{m(\alpha)}{\m_g} \mu_{\Thu}  \big(\big\{\eta\in\CC_K(\bSigma) \st  f(\eta)\le 1\big\}\big)\\
 & =\frac{m(\alpha) m(f) }{\m_g}. 
\end{align*}
This finishes the proof. 
\end{proof}

\end{document}